\documentclass[a4paper,11pt]{amsart}

\usepackage{color}

\theoremstyle{plain}
\newtheorem{thm}{Theorem}[section]
\newtheorem{theorem}[thm]{Theorem}

\newtheorem{lemma}[thm]{Lemma}
\newtheorem{corollary}[thm]{Corollary}

\theoremstyle{definition}
\newtheorem{remark}[thm]{Remark}

\newtheorem{defin}[thm]{Definition}

\newtheorem{example}[thm]{Example}

\numberwithin{equation}{section}

\newcommand{\sF}{{\mathcal F}}
\newcommand{\sG}{{\mathcal G}}

\newcommand{\sL}{{\mathcal L}}

\newcommand{\sT}{{\mathcal T}}
\newcommand{\sU}{{\mathcal U}}

\newcommand{\BP}{{\mathbb P}}

\newcommand{\Z}{{\mathbb Z}}
\newcommand{\PP}{\ensuremath{\mathbb{P}}}

\newcommand{\CC}{\ensuremath{\mathbb{C}}}

\newcommand{\hol}{\ensuremath{\mathcal{O}}}

\newcommand\la{\lambda}
\newcommand\s{\sigma}

\newcommand\Ga{\Gamma}
\newcommand\De{\Delta}
\newcommand\ga{\gamma}
\newcommand\de{\delta}

\newcommand{\ra}{\ensuremath{\rightarrow}}

\def\eea{\end{eqnarray*}}
\def\bea{\begin{eqnarray*}}

\newcommand\dual{\mathrel{\raise3pt\hbox{$\underline{\mathrm{\thinspace d
\thinspace}}$}}}
\newcommand\qe{\ifhmode\unskip\nobreak\fi\quad $\Box$}       

\def\BOX{\hfill\lower.5\baselineskip\hbox{$\Box$}}

\newtheorem{theo}{Theorem}[section]

\newtheorem{remarkk}[theo]{Remark}
\newenvironment{rem}{\begin{remarkk}\rm}{\end{remarkk}}

\newtheorem{prop}[theo] {Proposition}

\newtheorem{problem}[theo]{Problem}

\newcommand{\Proof}{{\it Proof. }}

\title [ On the double point formula]{The double point formula with isolated singularities and canonical embeddings}

\author{Fabrizio Catanese, Keiji Oguiso}

\address{Lehrstuhl Mathematik VIII, Mathematisches Institut der Universit\"{a}t
Bayreuth, NW II, Universit\"{a}tsstr. 30,
95447 Bayreuth, and Korea Institute for Advanced Study, Hoegiro 87, Seoul, 
133-722, Korea}
\email{Fabrizio.Catanese@uni-bayreuth.de}

\address{Mathematical Sciences, the University of Tokyo, Meguro Komaba 3-8-1, Tokyo, Japan, and Korea Institute for Advanced Study, Hoegiro 87, Seoul, 
133-722, Korea}
\email{oguiso@ms.u-tokyo.ac.jp}

\thanks{The first author acknowledges support of the ERC 2013 Advanced Research Grant - 340258 - TADMICAMT. The second author acknowledges support of JSPS Grant-in-Aid (S) 15H05738, JSPS Grant-in-Aid (B) No 15H03611. }
\keywords{Improper double points, double point formula, canonical models of surfaces, rational double points, isolated singularities,
Chern classes, Euler number, symplectic smoothing, tangent star, Gauss map}
\subjclass[2010]{14J29, 14J17, 14N15, 14E 99, 14B05, 32C22, 32Q40, 57R12, 57R42}

\begin{document}


\maketitle

\begin{abstract}
Motivated by the embedding problem of canonical models in small codimension, we extend  Severi's  double point formula to the case of
 surfaces with rational double points, and we give more general double point formulae  for  varieties with isolated singularities.
 
 A concrete application is for surfaces with geometric genus $p_g=5$:  the canonical model is embedded in $\PP^4$
  if and only if we have a complete intersection of type $(2,4)$ or $(3,3)$.
\end{abstract}

\addtocontents{toc}{\protect\setcounter{tocdepth}{1}}

\tableofcontents
\newpage

\section{Introduction}

Already for irreducible plane curves $C \subset \PP^2_{\CC}$ of degree $d$ with $\de$ ordinary double points (i.e., with
local analytical equation $xy=0$) we have two types of double point formulae.

 First there is 
the Pl\"ucker formula for the {\bf class} $ \omega_1 (C)$, which is a projective invariant, the number of lines $L$ passing through
a general point $O \in \PP^2$ and tangent to the curve $C$ at some smooth point:

$$  (I) \ \ \   \omega_1 (C) = d (d-1) - 2 \de ,$$ 

second, there is the genus formula:

$$ (II) \ \ \   2 g(C) - 2=    d (d-3)  - 2 \de ,$$
where 2g(C) is the first Betti number of the normalization $\tilde{C}$ of $C$, which is a topological invariant. 

The two formulae can be rewritten as expressing the number of double points   in terms of the degree and 
the class, respectively  in terms of
the degree and the genus:

$$   2 \de =   d (d-1) -   \omega_1 (C) = 
  d (d-3) + 2 - 2g(C).$$ 
Both formulae can be generalized also for plane curves with more general singularities, but the number $2 \de$
is replaced differently in each of the two formulae.

In the genus formula, the number $2 \de$ is replaced by the sum
$$  \sum_P m_P (m_P -1),$$
where the sum runs over all points $P$, also the  infinitely near ones, and $m_P$ is the multiplicity of
(the strict transform of) the curve $C$ in $P$.

In the Pl\"ucker formula appear not only the multiplicity of a singular point,
but also the local classes of the branches of $C$ at the singular points
(cf. \cite[Pages 277--282]{GH78}), and, for instance, in the case of an ordinary  {\bf cusp} singularity, i.e., when the local analytic equation of
$C$ is
$$ y^2 - x^3 = 0,$$
the contribution for the Pl\"ucker formula is $3$, the sum of the multiplicity (= 2) and the local class (= 1).

Another `philosophical' aspect of the genus formula is that it explains that a smooth projective curve of genus $g$ and degree $d$
cannot be projected isomorphically to the plane $\PP^2$ if $$2 g-2 - d (d-3) \neq 0.$$

It was classically known that every smooth subvariety $X' \subset \PP^N$ of dimension $n$ can be projected
isomorphically to $X \subset \PP^{2n+1}$, since the secant variety of $X'$ has dimension $2n+1$ and does not fill up
the whole space if $N \geq 2n+2$, so that we can project $X$ from a general point $O$ isomorphically
in $\PP^{N-1}$.   But when you want to project to $\PP^{2n}$ there appear, for a general
projection,  $\de$ {\bf Improper Double Points = IDP 's }, that is isolated singularities consisting of two smooth branches 
intersecting transversally.

If we put ourselves in a situation similar to the one of plane curves with only IDP's  as singularities (in this 
case IDP's are the same as before, ordinary double points), again we have two double point formulae.

The following Severi's double point formula,  established by  Severi in 1902, is a generalization of the Pl\"ucker formula:

$$ (I) \ \ \   2 \de  = d (d-1) - \sum_{i=1}^n  \omega_i (X),$$
and is given in terms of projective invariants $ \omega_i (X),$ called the {\bf ceti} of $X$ (see section 2).  

Later on, it was realized that these projective invariants are related to absolute invariants (similar to the genus $g$ of a curve), and this led to the theory
of characteristic classes, notably Chern classes. In terms of these, we have the Todd-Fulton-Laksov formula,
which is the modern double point formula:

$$ (II) \  \ \   2 \de  =  d^2  -    c_n (N_{f}) =  d^2  -   [ c(T_{X'} )^{-1} (1 + H) ^{2n+1}]_n,$$ 

where $f : X'\ra X$ is the normalization map, $N_f$ is the normal bundle to the immersion $f$,
defined through the exact sequence 
$$ 0 \ra T_{X'} \ra f^* (T_{\PP^{2n}} ) \ra N_{f} \ra 0  ,$$
 where we  denote by $H$ the pull back of the hyperplane class to $X'$, and where
 $[ c_0 +  c_1 + \dots + c_n ] _n$ denotes the part $c_n$ of dimension $2n$ of the sum in the cohomology ring (or in the Chow ring) of $X'$.

This double point formula has, as we saw,  two formulations (the second follows from multiplicativity of the Chern classes for exact sequences); the first has a topological flavour, and allows for generalizations: it says that $d^2  $,  the self intersection number of $X$ in 
$\PP^{2n}$,
equals to $ 2 \de$ plus the self intersection number of the zero section in the normal bundle $N_f$ (which maps via the exponential map
onto a neighbourhood of $X$). 

The second instead   lends   itself to concrete calculations, by which it is possible to exclude that a smooth variety
$X'$ of dimension $n$ can be embedded in a projective space of dimension $\leq 2n$ by certain linear systems.

 For instance, the first author has proven elsewhere \cite{canhyp} the following result:

\begin{theo}\label{divisorsonAV}
Let $D$ be an ample smooth  divisor in an Abelian variety $A$ of dimension $n+1$. If $D$ yields a polarization of type
$(m_1, \dots, m_{n+1})$, then the canonical  map $\Phi_D$ of $D$ is a morphism, and it can be an embedding only
if $p_g(D) : = h^0(K_D)  \geq 2n+2$, which means that the Pfaffian $m : = m_1 \cdot m_2  \dots \cdot  m_{n+1}$ satisfies the
inequality $$ m \geq n + 2.$$
\end{theo}

The main  idea for proving the above theorem is,   assuming that $D$ is embedded by $H^0(K_D) = H^0(\hol_D(D))$ and that
$N : = p_g(D) -1 \leq 2n$, to  apply the double point formula (in the special case $\de = 0$)  to $D$ and
to  its iterated hyperplane sections.  See also \cite{deV75}, \cite{Ka19} for other applications.

The  main motivation for our work stems from surface theory, namely the desire to extend the following theorem
of the first author \cite{Ca97}:
\begin{theo}\label{SantaCruz}
Let $S$ be a complex projective surface with  ample canonical divisor and with $ p_g(S) = 5$. Then the canonical map is an embedding
if and only  $S$ is a smooth complete intersection in $\PP^4$ of type $(2,4)$ or of type $(3,3)$.
\end{theo}

If we take instead a minimal surface $S$ of general type with $ p_g(S) = 5$, the canonical map factors
through the canonical model $S'$ of $S$, a normal surface with canonical singularities (rational double points),
and we may ask the similar question: when do we get an embedding of the canonical model? The answer
is analogous:

\begin{corollary}\label{cor1} Let $S$ be a minimal smooth surface of general type with $p_g(S) = 5$ such that the canonical map  
$$\Phi_{|K_S|} : S \to S'\subset \BP^4$$ is a morphism with
 image isomorphic to  the canonical model $S'$ of $S$. Then, setting  $d : = (K_S^2)$,
$$12 \chi({\mathcal O}_S) = (17-d)d\,\, .$$
In particular,  $S'$ is then a complete interesection of type $(2, 4)$ or $(3, 3)$, if either the base field is of characteristic $0$ or the characteristic of the field $k$ is $\neq 2$ and $h^1({\mathcal O}_S) = 0$. 
\end{corollary}

The idea is to extend the double point formula to the case of  surfaces with singularities either  IDP's or canonical singularities (and then setting $\de=0$
for the application): this is done
in an elementary way in  Theorem \ref{thm1},  Section 5  over any algebraically closed field of any characteristic. Then one establishes the above numerical formula, by which follows  that $d=8,9$;  then the proof proceeds exactly as in  Theorem \ref{SantaCruz}.

Since it would be interesting to extend this type of result also in higher dimension, it is desirable to extend the
double point formula to varieties  $X$ with singularities.  In this paper we restrict 
to varieties with isolated singularities.

Fulton and Laksov's generalization \cite{FL77} allows in particular  the variety $X$, outside the IDP' s,
to have local complete intersection singularities, a property which implies that $X$ is locally smoothable at these points.

In this flavour we prove in the last section another generalization, of which the following theorem is a special case:

\begin{theo}\label{sympl}
Let $X \subset \PP^{2n}$ be a variety with isolated singularities, of which $\de$ are Improper Double Points,
and the other are normal singularities admitting a smoothing.  Then $X$ admits a global smoothing to a
symplectic immersed manifold $M \subset \PP^{2n}$, with exactly  $\de$  Improper Double Points,
and we have, if $f : M' \ra M$ is the immersion,
$$ d^2 = 2 \de + e(N_f) ,$$
where $e$ is the Euler class of the oriented normal bundle to the map.
\end{theo} 

In spite of their elegance, the two above generalizations have the drawback that the top Chern class of the normal bundle
of the partial resolution of $X$ at the IDP's (respectively the Euler class in  Theorem \ref{sympl}) is not directly computable (at least this is our impression). This is due to the fact that Chern classes are multiplicative only for sheaves on smooth varieties, or, for   general varieties, only  for exact sequences of vector bundles, 
while Euler classes are not multiplicative.

It seems therefore worthwhile to go back to Severi's approach. 

Severi established an inductive formula,

$$ 2  \hat{\de} = 2 \de + \omega_n (X),$$
where $\hat{X}$ is obtained from $X$ intersecting with a general hyperplane $H$, and then projecting $X \cap H$
from a general point
$O \in H$ to $\PP^{2n-2}$, to obtain  $\hat{X}$ with $ \hat{\de} $ IDP's.

First of all we explain in  Section 2 how the two different double point formulae are related to each other, showing how to formulate Severi's inductive formula in terms of the Gauss map
of the immersion $ f : X' \ra \PP^{2n}$.

Indeed, if $Q_{X'}$ is   the pull back of the universal quotient bundle of the Grassmannian $ Gr(n,2n)$ (of projective subspaces of 
dimension $n$ in $\PP^{2n}$), then we have:
$$ \omega_n(X) = c_n (Q_{X'}) ,$$
and the relation between the two approaches is given by the simple formula
$$ Q_{X'} = N_f (-1).$$

For worse singularities, we take a resolution $Y$ of $X$ on which the Gauss map becomes a morphism.
Then $f : Y \ra \PP^{2n}$ is no longer an immersion and  the two sheaves are different, yet there is 
a surjection $N_f (-1) \ra Q_Y$, and the double point formula can be expressed in terms of the kernel sheaf
$\sF$, called {\bf discrepancy sheaf}, see  Theorem \ref{Discrepancy}.

This is done in  Section 4, using the results of   Section 3 where we generalize Severi's double point formula by
extending Severi's inductive formula almost verbatim, but using the notion of {\bf  the tangent star}
of a singular point $x$: this  is the closure of the union of the lines which are 
  limits of secant lines joining a pair of points $x_1, x_2$
tending to $x$.

\begin{theo}\label{SG}
Let $X \subset \PP : = \PP^{2n}$ be an irreducible nondegenerate subvariety of dimension $n$ having
isolated singularities $x_1, \dots, x_h$. 

Assume that these singularities are either Improper Double Points or have  tangent star of dimension $\leq 2n-1$.

Let $\de$ be the number of the IDP ' s:
then, in the notation of Theorem \ref{SeveriT}, we have

$$ 2  \hat{\de} = 2 \de + \omega_n (X).$$ 

\end{theo}

Of course our above partial results raise several questions. 

First of all, in  Theorem \ref{SG} we observe that IDP's have full  tangent star (the tangent star  is the whole space).

There is an obvious generalization of  Theorem \ref{SG}, for varieties with isolated singularities: but the main problem
here is to calculate exactly the contribution of the points with full tangent star
to the double point formula. This problem is related to work of   Flenner, O' Carroll and Vogel and others, who defined a cycle which gives (as far as we understand) the contribution to this multiplicity
see \cite{FCV}.

This said, the topic of double point formulae can been treated in so many different ways, that it is a very challenging and currently
active area of research 
to investigate the relations among the several approaches: 
our approach here was led by the problem we wanted to solve, and by the search of effectively computable formulae.

The expert reader may notice that the resolution of indeterminacy   of the Gauss map was
later called  Nash blowing up, and used by Mather to define Chern classes of singular varieties \cite{Na95}, \cite{Ma74}.
However, Mather-Chern classes are defined by taking the push forward of the Chern classes of the pull back (under the Gauss map) 
of the universal
subbundle, whereas we need those of the quotient bundle to be pushed forward, and again we are missing multiplicativity.
MacPherson \cite{Ma74}  introduced functorial Chern classes of singular varieties; Piene \cite{piene}  obtained expressions relating the polar classes to the Chern-Mather classes, and  used these to calculate local Euler obstructions of hypersurfaces  in terms of Milnor numbers. Recently Aluffi related the MacPherson Chern classes to the Fulton Chern classes first \cite{Al94} in the case of hypersurfaces, and later, \cite{Al18}, in  greater generality. There has been other related recent work, such as \cite{Al16}, \cite{bss}.

It would be interesting to see whether some of these ideas shed some light on the problem of  embeddings of canonical models in higher dimension,
but we postpone this investigation to the future.

\section{The classical double point formula}

The classical double point formula established by  Severi in 1902 concerns the following situation:

$X$ is an irreducible subvariety of dimension $n$ in $\PP^{2n}$, whose  singularities are only {\bf Improper Double Points = IDP 's } , that is isolated singularities consisting of two smooth branches 
intersecting transversally.

Letting $X'$ be the normalization of $X$, $X'$ is smooth projective, and we have a normalization map $\nu : X' \ra  \PP^{2n}$ which is an immersion, hence we have the Gauss map
$$ \ga :  X' \ra Gr (n,2n) $$
associating to $x \in X'$ the projective linear subspace image of the tangent space $T X'_x$ under the derivative $D \nu_x$ of $\nu$ at $x$. 

\begin {defin}
The n-th {\bf ceto} $\omega_n (X)$ of  $X \subset \PP^{2n}$  is the number of subspaces $ \ga (x)$  passing through a general point $ O \in \PP^{2n}$.

More generally, for any projective variety  $X$ of dimension $n$ in $\PP^N$,  $X \subset \PP^N$   , the n-th {\bf ceto} of $X$, $\omega_n (X)$, is the number of linear subspaces of dimension $n$
tangent to $X$ at smooth points, and intersecting a general linear subspace of codimension $2n$ in $\PP^N$.

The i-th {\bf ceto} of $X$, $\omega_i (X)$,  is the i-th {\bf ceto} of the intersection $X \cap  L'$ of $X$ with a general subspace $L'$ of codimension $n-i$.
\end{defin}

The following is Severi's main assertion (\cite{Se02}, see also \cite{Ca79} for a more detailed proof) 

\begin{theo} {\bf (Severi' s Inductive Statement)}\label{SeveriT}
Let $X \subset \PP^{2n}$ be a subvariety of dimension $n$ having  only $\de$  IDP' s as singularities, let  $ O \in \PP^{2n}$ be a general point,
let $ L$ be a general linear subspace of $ \PP^{2n}$ of codimension $2$, let $H$ be the hyperplane spanned by $O$ and $L$.

Then, denoting by $\hat{X'}$ the (smooth)  intersection $X \cap H$, the projection with centre $O$ to $L$ of $\hat{X'}$ is a variety $\hat{X}$
with only $\hat{\de}$  IDP' s as singularities, and we have 

$$ 2  \hat{\de} = 2 \de + \omega_n (X).$$ 

\end{theo}

Using the inductive statement, and observing that, for a plane curve of degree $d$ and with $\de'$ double points, the 1-ceto is just its class, the number of tangents through a general
point, hence equal to $d (d-1) - 2 \de '$, Severi obtains the following result:

\begin{theo} {\bf (Severi' s Double Point Formula)}
Let $X \subset \PP^{2n}$ be a subvariety of dimension $n$ and degree $d$ having  only $\de$  IDP' s as singularities:

 then we have 

$$ 2 \de  = d (d-1) - \sum_{i=1}^n  \omega_i (X).$$ 

\end{theo}

The modern  version of the double point formula uses holomorphic vector bundles and Chern classes, and can be geometrically explained as follows,
over the complex numbers.

Define the normal bundle $N_{\nu}$  to the map $\nu$ through the exact sequence
associated to the derivative of $\nu$

$$ 0 \ra T_{X'} \ra \nu^* (T_{\PP^{2n}} ) \ra N_{\nu} \ra 0  .$$

The exponential map yields an oriented submersion of a neighbourhood of the $0$-section of  $N_{\nu}$ (which we identify to $X'$) onto a neighbourhood of $X$,
and one can find a small perturbation of the $0$-section yielding an oriented submanifold $X'_t$  intersecting  $X'$ transversally  in a finite number of points,
so that the self intersection $(X')^2$ of $X'$ in $N_{\nu}$ equals the Euler class of the bundle $N_{\nu}$, in turn equal to the top Chern class $c_n (N_{\nu})$.

It is easy to see that the self intersection of $X$ equals  $X^2 = 2 \de + (X')^2$: since $X$ has degree $d$, we get 
(see \cite{LMS75} for an algebraic proof)

\begin{theo} {\bf (Modern Double Point Formula)}
Let $X \subset \PP^{2n}$ be a subvariety of dimension $n$ and degree $d$ having  only $\de$  IDP' s as singularities:

 then we have

$$ (DP) \ d^2 =  2 \de +   c_n (N_{\nu}) = 2 \de + c_n ( \nu^* (T_{\PP^{2n}} ) - T_{X'} ) .$$ 

\end{theo}

Using then the    multiplicativity of Chern classes for exact sequences, and that the total Chern class $ c (T_{\PP^{2n}} ) = c(\hol (1))^{2n+1}$,
if we denote by $H$ the pull back of the hyperplane class to $X'$, we get

$$2 \de  =  d^2  -   [ c(T_{X'} )^{-1} (1 + H) ^{2n+1}]_n,$$
where $[ c_0 +  c_1 + \dots + c_n ] _n$ denotes the part $c_n$ of dimension $2n$ of the sum in the cohomology ring (or in the Chow ring).

 The most  interesting (not only historical) question is then: how are the two formulae related ?

One deals with projective invariants of $X$, the other with invariants of the abstract variety $X'$, but indeed the Gauss map $\ga$ produces
some vector bundles on $X'$.

Denote as usual by $V$ the complex vector space such that $\PP (V) = \PP^{2n}$ (here $\PP (V) $ is the variety of 1-dimensional subspaces of $V$):
then on the Grassmann variety $G : = Gr (n,2n) $ we have an exact sequence of vector bundles

$$ 0 \ra U \ra V \otimes \hol_G  \ra Q \ra 0,$$

where $U$ is the universal subbundle, whose fibre at a subspace $W \subset V$  is tautologically  equal to  $W$,
while $Q$ is the quotient bundle.

In our situation, a point $O \in \PP : = \PP (V)$ is the class $[v]$ of a vector $v \in V \setminus\{0\}$, and $v$ defines a regular section of  $V \otimes \hol_G$,
and, by composition, $s_v : \hol_G \ra Q$. The section $s_v$ vanishes exactly at the subspaces $W$ which contain $v$.

Taking the pull-back under the Gauss map, we get 

$$ 0 \ra \ga^* U \ra V \otimes \hol_{X'}  \ra \ga^* Q \ra 0,$$
and we see then that $\omega_n(X)$ equals to the numbers of zeroes of $\ga^*(s_v)$, hence

$$ \omega_n(X) = c_n (\ga^* Q).$$

The Euler sequence in $\PP (V) = : \PP$  pulls back to $X'$ yielding 

$$ 0 \ra \nu^* \sU \ra  V \otimes \hol_{X'}  \ra \nu^*  T_{\PP} (-1) \ra 0 .$$ 

Here $\sU \cong \hol_{\PP} (-1)$ is the tautological subbundle on $\PP$, and an easy but important observation
is the inclusion

$$  \nu^* \sU \subset \ga^* U \  .$$

Denoting by $\hol_{X'} (H) : = \nu^*  \hol_{\PP} (1)$, the derivative $ D \nu$ yields, after twisting,  the   exact sequence 

$$ (*) \ 0 \ra T_{X'} (-H) \ra \nu^* (T_{\PP}(-1) ) \ra N_{\nu} (-H) \ra 0  .$$

Moding out the other  two previous exact  sequences by $\nu^* \sU$, we get 

$$ (**) \ 0 \ra (\ga^* U) / (\nu^* \sU)  \ra \nu^* (T_{\PP}(-1) )   \ra \ga^* Q \ra 0.$$

Since $\nu$ is an immersion, $D \nu_x : TX' _x \ra  (\ga^* U / \nu^* \sU)_x$
is an isomorphism, hence $(*)$ and $(**)$ are the same exact sequences.  In particular, we have an isomorphism $\ga^* Q \simeq N_{\nu} (-H)$.

Restricting $N_{\nu}$ to a hyperplane section $ H = \hat{X'}$, we obtain:

$$  0  \ra N_{\nu} (-H) \ra   N_{\nu}  \ra  N_{\nu}| \hat{X'} : = N_{\nu} \otimes \hol_{\hat{X'} } \ra 0  ,$$
yielding $ c ( N_{\nu} ) = c ( \ga^* Q ) \cdot c ( N_{\nu}| \hat{X'} ) $  for the total Chern classes.

The upshot here is: the restriction $N_{\nu}| \hat{X'}$ is the normal bundle $N_{\nu'}$
of the map $\nu ' :  \hat{X'} \ra \PP^{2n-1}$.

On the other hand, projection from the point $O$ yields a map $$\hat{\nu} :  \hat{X'} \ra \PP^{2n-2}$$
such that 
$$N_{\nu'} = N_{\hat{\nu}} + \hol_{\hat{X'}}  (H)$$
  in the K-group, hence 
$$c( N_{\hat{\nu}}  ) = c ( N_{\nu'} ) (1 + H)^{-1}{\color{red} .}$$
in the Chow ring.
We can then calculate 
$$  c_n ( \ga^* Q ) =  c_n ( N_{\nu} (-H) )= c_n ( N_{\nu})  -H c_{n-1} ( N_{\nu}) + \dots + (-1)^n H^n  $$
$$ =  c_n ( N_{\nu})  - H ( c_{n-1} ( N_{\nu} -  \hol_{X'} (H))) =  c_n ( N_{\nu}) -  i_* (c_{n-1}  ( ( N_{\nu'} -  \hol_{\hat{X'}} (H))))=$$
 $$=  c_n ( N_{\nu}) -  i_* (c_{n-1}  ( N_{\hat{\nu}} ))
 $$

where  $ i : \hat{X'} \ra X'$ is 
the natural inclusion.

We have thus shown   
$$   \omega_n (X) = c_n  ( N_{\nu} )   -  i_* (c_{n-1}  ( N_{\hat{\nu}}  )),$$ 
which, using formula $(DP)$  for the double points, gives Severi's inductive formula.

\section{ Extension of Severi's Statement}

In this section $X \subset \PP : = \PP^{2n}$ is an irreducible nondegenerate subvariety of dimension $n$ having
isolated singularities. We set $ X^* : = X \setminus Sing(X)$, and let $x_1, \dots, x_h$ be the singular points of $X$.

Severi's basic idea starts with the consideration of the secant map
$$ \s:  (X \times X) \setminus \De_X  \ra \sG : = Gr (1,2n)$$
associating to  $x\neq x'$ the line $ \s (x,x') : = x * x'$.

Taking the closure $\Sigma$ of the graph of $\s$  in $(X \times X) \times \sG$, and on it the pull-back of the projectivized 
universal subbundle $\sL_{\sG} \subset \sG \times \PP$ of the Grassmannian,
we obtain
a line incidence relation
$$ \sL_X \subset \Sigma \times \PP \subset (X \times X) \times \sG \times \PP,$$
which is a $\PP^1$-bundle over $\Sigma$.

There is a  natural surjective projection $ p : \sL_X \ra  (X \times X) $ and a natural
projection $\pi : \sL_X  \ra  \PP$.

\begin{defin}

We define as customary (cf. \cite{Jo78})  for $x \in X$,
$$ \sL_x  : = p^{-1} ( \{ (x,x)\}), $$
and define the  tangent star  of $X$ at $x$ as
$$ Star_x (X) : = \pi (\sL_x).$$
\end{defin}

For  any smooth point $x \in X$, the tangent star is the projective tangent space of $X$ at $x$,
and in general  one has a series of inclusions: 
$$ \sT_x (X) \subset Star_x (X) \subset T^{Zar}X_x,$$
in words: the tangent cone at $x$ is contained  in the tangent star at $x$, which is contained in the 
projective closure of the Zariski tangent space.

Since $\Sigma$ has dimension $n$, the fibre $\Sigma_x$ over $(x,x)$  has dimension at most $2n-1$,
hence in general the  tangent star at $x$ has dimension at most $2n$. Moreover, $\Sigma_x$ consists
of lines passing through $x$, hence

\begin{prop}
The tangent star is a cone with vertex $x$. If $X \subset \PP : = \PP^{2n}$ has dimension $n$, then the following assertions (1), (2), (3) are equivalent: 
\begin{enumerate}
\item
the tangent star at $x$ equals  $\PP^{2n}$;
  
\item
$\Sigma_x$ has dimension $2n-1$;   
 \item
 $\Sigma_x$ is the Schubert cell $C_x$  of all lines
through $x$.
\end{enumerate}

In particular the degree of the map $\sL_x \ra  \PP^{2n}$ is always either $0$ or $1$.
\end{prop}
\Proof
Assertion (1) amounts to $ dim (Star_x(X)) = 2n$, which implies $ dim (\Sigma_x) \geq 2n-1$, 
which is equivalent to  assertion (3) since $\Sigma_x \subset C_x = \PP^{2n-1}$, where $C_x = \PP^{2n-1}$
is the variety (Schubert cell)  of lines through $x$. Assertion (3) obviously implies (1), and also the last assertion is clear, since
the degree is $0$ if and only if the map is not surjective.

\qed

\begin{rem}.

(i) The IDP 's are isolated singularities with {\bf full tangent star}, i.e., whose tangent star is the whole space $\PP^{2n}$.

(ii) Each point of embedding dimension $< 2n$ (that is, the Zariski tangent space has dimension $\leq 2n-1$ {\color{red} )}
has a non full tangent star.

(iii) Observe that $\sL_x$ is taken here with the reduced structure: in general it is part of the ramification
locus of $\pi$.

(iv) see \cite{SUV97} for some related results concerning tangent stars of complete intersections.

\end{rem}

We can now prove a generalization of Severi's theorem  (Theorem \ref{SeveriT}, the inductive assertion):

\begin{theo}\label{SG}
Let $X \subset \PP : = \PP^{2n}$ be an irreducible nondegenerate subvariety of dimension $n$ having
isolated singularities $x_1, \dots, x_h$  only. 

Assume that these singularities are either Improper Double Points or have  tangent star of dimension $\leq 2n-1$.

Let $\de$ be the number of the IDP ' s:
then, in the notation of Theorem \ref{SeveriT}, we have

$$ 2  \hat{\de} = 2 \de + \omega_n (X).$$ 

\end{theo}

\Proof 
Let $Z$ be the image of $\Sigma$ inside the Grassmannian $\sG$. Observe that $ q : \Sigma \ra Z$
factors through the  involution $\tau$ on $\Sigma$ birationally induced by the involution exchanging
the factors of $X \times X$.

Argueing as in Lemma 1 of \cite{Ca79},
we see that $\Sigma \ra Z$ is generically finite of degree $2$, hence $Z$ is irreducible of dimension $2n$,
and birational to $\Sigma / \tau$.

Given a point $O \in \PP$, we let $C_O$ be the Schubert cell of the lines through $O$,
and we consider the curve $$C : = q^{-1} (Z \cap C_O) \subset \Sigma \subset (X \times X) \times \sG.$$ 

Denote by $\Ga_2$ the projection of $C$ inside $X \times X$, respectively by $\Ga$ its projection to $X$
(because of $\tau$ symmetry, the first and second projection give the same result).

Observe that, by our assumption, for general choice of $O$,  $\Ga$ does not pass through
 the singular points of $X$ which are not IDP 's. 
 
 Let $X'$ be the normalization of $X$ precisely  at the IDP 's. 
Then we can construct an incidence variety 
$\Sigma' \subset (X' \times X') \times \sG$ which is birational to $\Sigma$.
Hence we can similarly define curves $C', \Ga'_2, \Ga'$.

The proof proceeds now  exactly as in \cite{Ca79}, pages 766-773 (but with a slightly different notation).

Step I)  Proposition 2 shows that $\Ga_2'$ is a smooth curve, and the graph of  a birational involution $\tau$ on $\Ga'$
such that, if $x'_i, x''_i$ are the two points lying over an IDP $x_i$, then $\tau(x'_i) = x_i''$.

Step II) For general choice of $ O, L$, setting $ H = O * L$, $\hat{X}' = X \cap H$ is smooth and
$\hat{X} \subset L \cong \PP^{2n-2}$ has only IDP' s. Moreover $\Ga \cap L = \emptyset$, 
$H$ intersects $\Ga$ in $deg (\Ga)$ distinct 
smooth points of $\Ga$,  and different from the points
$Q_1, \dots , Q_{\omega_n(X)}$ \footnote{ O is chosen such that $C_O$ intersects  $\overline{\ga(X')} $
transversally at the points $\ga(Q_i)$} where the tangent space to $X$ does not pass through $O$, so that in particular
$\tau$ has no fixed points on $H \cap \Ga$. Denote these points by $R_1, \dots, R_{2m}$.
We observe that $ m = \hat{\de}$: since the points in $\Ga \cap H$ contribute in pairs to IDP 's of $\hat{X}$.

Step III )  Let $\phi_0, \phi_1$ be linear forms such that $H = \{ \phi_0 = 0\}$, $L = \{ \phi_0 = \phi_1 =  0\}$,
and let $\phi : \PP \setminus L \ra \PP^1$ the projection with centre $L$.

Then we have a morphism $\psi : \Ga_2' \ra (\PP^1 \times \PP^1)$, the composition of the projection into 
$X' \times X'$ with the map $\phi \times \phi$.

Step IV) The points $Q_i$ are the  fixed points of the involution $\tau$, hence if $\De$ is the diagonal
in $\PP^1 \times \PP^1$, then 
$$\psi^{-1} (\De) = \{ (R_i, \tau R_i), (P'_j,P''_j), (P''_j,P'_j),(Q_h,Q_h) \}.$$

As proven on page 772 of \cite{Ca79} the points $\psi(Q_h,Q_h)$ are distinct, equivalently,
the points $\phi (Q_h)$ are distinct, in particular $\psi$ yields a birational map from $\Ga'_2 $
to its image $\Ga^0$. The symmetry $\tau$ on $\Ga'_2 $ induces the coordinate exchange
on $\Ga^0$, so that $\Ga^0$ is symmetric, and of bidegree $(2 \hat{\de}, 2 \hat{\de})$,
since  $\Ga \cap H$ consists of  $2 \hat{\de}$) smooth points.

Hence the intersection number of $\Ga^0$ with $\De$ equals  $4 \hat{\de}$; 
but the previous calulation of the inverse image of the diagonal on the smooth curve $\Ga_2'$
shows

$$4 \hat{\de} = 2 \hat{\de}  + 2 \de + \omega_n (X) \Leftrightarrow 2 \hat{\de} =  2 \de + \omega_n (X) {\color{red} .}$$ 

\qed

\begin{rem}
The above proof works  for any  algebraically closed field of characteristic $\neq 2$.

\end{rem}

\begin{remark}
Assume that $X \subset \PP : = \PP^{2n}$ is an irreducible nondegenerate subvariety of dimension $n$ having
isolated singularities $x_1, \dots, x_h$. 

  We have used here Severi's original approach
rather than subsequent developments, as done in \cite{FCV} or \cite{FM}, where the authors  used the St\"uckrad-Vogel cycle.

It seems to us that from these works follows: one can determine,  for  each singular point $P = x_i$, a `number of double points' $\de_P$ (clearly  such that
$\de_P = 0$ if $Star_P(X)$ has dimension $\leq 2n-1$),  with the property  that, defining
$$\de  = \sum_{P \in X} \de_P,$$
we have, in the notation of Theorem \ref{SeveriT}, 

$$ 2  \hat{\de} = 2 \de + \omega_n (X).$$ 

The main question is how to determine these numbers explicitly.
\end{remark}

\begin{example}
Let $X \subset \PP^4$ be the projective cone with vertex $x_0$ over the twisted cubic curve $D\subset \PP^3$.
Then $x_0$ is the only singular point of $X$, and clearly  its tangent star is $\PP^4$. 

Here $D$ projects to a nodal plane cubic, hence $\hat{\de} = 1$.

Here $\omega_2(X) = 0$, since the union of the tangent planes to $X$ has dimension $3$. Hence $\de=1$
and the normal point $x_0$ contributes $1$ to the number of improper double points.

Replacing $D$ by a non degenerate smooth curve of genus $g$ and degree $d$ in $\PP^3$, we see that again 
$\omega_2(X) = 0$, hence  $2 \de_{x_0} = (d-1) (d-2)  -2g $.

\end{example}

\section{Comparing normal sheaf and quotient bundle }

As in the  preceding section (where we have computed    the   number $\de$ of IDP' s )
we consider    irreducible subvarieties $ X \subset \PP := \PP^{2n}$ of dimension $n$ and
with isolated singularities,  which are either IDP 's or have degenerate tangent star (i.e., of dimension $\leq 2n-1$).

In this section we slightly change our notation and denote the normalization of $X$ by $X'$.

In the case where $X$ is a surface, we shall take $\tilde{X} = S$ to be a minimal resolution of singularities of $X'$,
but in general $\tilde{X} $ shall be a resolution of singularities of $X'$, and $ Y \ra \tilde{X} $ a further
resolution such that the Gauss map becomes a morphism $$ \ga : Y \ra G = Gr (n,2n).$$

We let $ f : Y \ra \PP$ be  the morphism composition of the two maps $ Y \ra X' \ra X \subset \PP$, and we shall use the following
selfexplanatory notation: for instance  $ Q_Y = \ga^* Q$ .

 We set  $\hol_Y (H) : = f^*  \hol_{\PP} (1)$. Then, as in Section 2, 
$$ (*) \ 0 \ra T_Y (-H) \ra f^* (T_{\PP}(-1) ) \ra N_f (-H) \ra 0  ,$$

and 

$$ (**) \ 0 \ra U_Y /  \sU_Y  \ra f ^* (T_{\PP}(-1) )   \ra Q_Y \ra 0.$$

Since $f$ is generically an immersion, $D f$ yields a sheaf 
inclusion $T_Y (-H) \subset U_Y /  \sU_Y $ and
associated exact sequences
$$0 \ra T_Y (-H) \ra  U_Y /  \sU_Y  \ra \sF \ra 0 ,$$ 
$$ (***) \ 0 \ra \sF   \ra  N_f (-H)  \ra  Q_Y \ra0 ,$$ 
where the sheaf $\sF$ is supported at the inverse image of the  (finitely many) singular points of $X$
which are different from the IDP' s.

\begin{defin}
We define the sheaf $\sF$ on $Y$ to be the IDP discrepancy sheaf on $Y$. 

\end{defin}

Note that $Y$ is not unique, since not always there is a minimal resolution of the closure of the Gauss map.
\begin{theo} \label{Discrepancy}
Let $X \subset \PP^{2n}$ be a subvariety of dimension $n$ and degree $d$ having  only isolated singularities,
of which $\de$ are IDP' s, and the rest are singular points with degenerate tangent star
and whose associated  discrepancy sheaf  $\sF$  on $Y$ is equal to zero
$\sF = 0$. 

 Then we have 
$$ d^2 =  2 \de +   c_n (N_f ) = 2 \de + c_n ( f^* (T_{\PP^{2n}} ) - T_{Y} ) .$$ 

 If instead $\sF \neq 0$ on $Y$, we have 

$$ d^2 =  2 \de +   c_n (N_f ) + c_n (N_f (-H) - \sF) - c_n (N_f (-H))$$
or, equivalently,
$$ d^2 =   2 \de +   c_n (N_f ) +
[c (N_f (-H)) \cdot c(  \sF)^{-1}]_n - c_n (N_f (-H)) .$$

\end{theo}
\Proof
The proof of the first assertion follows from Theorem \ref{SG}, and
 the argument given in  Section 2 to show how to relate $c_n(Q_Y)$
with  $c_n (N_f )$.

The proof of the second assertion follows with the same derivation as above
 plugging in  the relation $c(Q_Y) = c(N_f (-H) ) c(\sF)^{-1}$.

\qed

Of course, for most applications, one has to determine the sheaf $\sF$; this can be done by explicit but sometimes lengthy
calculations. 

For instance, in the  case of isolated hypersurface singularities  $$X = \{ x | F(x) = 0 \}  \subset W : = \CC^{n+1} ,$$
the tangent plane at  the smooth points $ x \in X^* := X \setminus \{0\}$ is the dual (annulator) subspace to the gradient at $x$.
Composing the gradient map with the projection onto $\PP(W^{\vee})$ (the projective space of lines in the vector space
$W^{\vee}$), 
$$  \nabla F :   X^* \ra W^{\vee} \setminus \{0\} , \ p : W^{\vee} \setminus \{0\} \ra \PP(W^{\vee}) = : \PP\,\, ,\,\, f : = p \circ \nabla F  $$ 
and letting $Z$ be the closure of the graph of $f$, we see that 
 the pull back $Q_Z$ of the quotient bundle is locally the direct sum of a trivial summand with 
$f^* (\hol_{\PP} (1))$.

Indeed, if $L \cong \hol_{\PP} (-1)$ is the universal subbundle on $\PP$,  the exact sequence 
$$ 0 \ra L \ra  W^{\vee}  \otimes \hol_{\PP}  \ra T_{\PP} (-1) \ra 0$$
dualizes to
$$ 0 \ra L^{ann} = \Omega_{\PP}^1(1)  \ra  W  \otimes \hol_{\PP}  \ra L^{\vee} =   \hol_{\PP} (1) \ra 0.$$
For simplicity of notation we continue to denote by $f$ the map of $Y$ to $\PP(W^{\vee})$ obtained from
the resolution $ Y \ra Z$.

Since $N_f$ is the cokernel of $T_Y \ra f^* (W  \otimes \hol_{\PP}) $,  we have exact sequences:
$$ 0 \ra \sF \ra N_f \ra  f^* (\hol_{\PP} (1)) \ra 0,$$ 
$$ 0 \ra T_Y \ra f^* (\Omega_{\PP}^1(1)) \ra \sF \ra 0.$$

The second sequence shows that the discrepancy sheaf $\sF$ is a Cohen-Macaulay sheaf of codimension $1$ on $Y$,
supported on an exceptional divisor.

 The first sequence instead identifies, in a neighbourhood of this divisor,  the virtual sheaf $N_f (-H) - \sF$ as a line bundle: $f^* (\hol_{\PP} (1))$.

\begin{example}
Consider a hypersurface singularity in $\CC^3$, of multiplicity equal to $d$, and with tangent cone equal to the cone over a smooth plane curve 
$C$.

Let the local equation be
$$ X : = \{ x =(x_1,x_2,x_3) | F (x) = 0, \ F(x) = P_d(x) + \dots \},$$ 
so that the blow-up $Y$ of $X$ at the origin is smooth,
$$ Y \subset \CC^3 \times \PP^2 , \ Y  = \{ (x, u)) |  x = \la u, \la^{-d} F (\la x) = 0 \},$$
and the Gauss map is a morphism on $Y$, which we view as a hypersurface in the universal subbundle over $\PP^2$.
Here $C = \{ u| P(u)=0\}$ is identified to the execptional curve of the blow-up, and on $Y$ we have $C^2 = -d$.

Here $U / \sU$, at the points of $C$,  equals the cone $U'$ over   the tangent line $T_uC$. If $\phi(t) $ is a local parametrization
of $C$, $ (t,\la) \mapsto \la \phi(t)$, hence the image of the tangent space to $Y$ is the subspace generated by $\la \phi'(t)$ and $\phi(t)$,
while $U'$ is generated by $ \phi'(t)$ and $\phi(t)$. Hence the cokernel $\sF$ is a line bundle on $C$,
equal to $TC(C) = \hol_C((d-2)C)$.

Therefore $c(\sF)^{-1} = 1 - C + C (d-2) C$, and the contribution to $ 2 \de$ equals 
$$ - c_1(N_f) (-C) - (d-2) C^2=  K_Y \cdot C + d (d-2) = - (d-2) C^2 + d (d-2) = 2 d (d-2). $$  
\end{example}

The previous calculation shows that we get a positive contribution to $\de$, equal to $d (d-2)$: this contribution is equal to $0$
if and only if we have a double point.

\begin{remark}
(a) Similar calculations could be done in order to show that the contribution we obtain equals $0$ for all rational double points, which are 
hypersurface singularities with equation $ z^2 = f(x,y)$, where $f$ has multiplicity at most $3$, and no infinitely near triple point.

It is to observe that $Y$ is not here a minimal resolution of the singularity.

We do not pursue this calculation here, since we shall give  an elementary  proof in a later section that rational double points do not contribute to
$\de$.

(b) In the previous  example, the Milnor number of the singularity equals $(d-1)^3$, while the topological Euler characteristic of 
the exceptional divisor $C$ is $ e(C) = - d (d-3)$, hence $e(Y) - e(X) = - [ d(d-3) + 1]$.

Indeed, in this case the contribution equals $  2 [(d-1)^2 - 1]$ which is twice the Milnor number $\mu_1$ of a general hyperplane section
 through the singular point, diminished by $1$. This reminds of a formula by Piene in \cite{piene}, end of page 25; however the number
 $ 1 - \mu_1$ is not zero for rational double points $z^2 - x^3 - y^3 =0$, or $z^2 - x^3 - y^5 =0$ (it is $-1$). This shows that
 our contribution is not the local Euler obstruction.

\end{remark}

\section{Surfaces with rational double points}

In this section we work over  any algebraically closed field  $k$  of characteristic $p \ge 0$  and we show the following theorem for surfaces:

\begin{theorem}\label{thm1} Let $S$ be a smooth projective surface and let
$$\mu : S \to X \subset \BP^4$$
be the minimal resolution of  a surface  $X$ whose singularities are  only $\delta := \delta (X)$  IDP's,  and some rational double points.
Then: 
$$2 \delta = d^2 - 10d + 12 \chi({\mathcal O}_S) - 5(H_S.K_S)_S -2(K_S^2)_S\,\, .$$
Here $H_S = \mu^*H$ for the hyperplane class $H$ of $\BP^4$ and 
$d = (H_S^2)_S = {\rm deg}\, X$. 
\end{theorem}

Theorem \ref{thm1} will be applied in the following natural situation. 

Let $S$ be a smooth minimal surface and let $f : S \to S'$ be either the morphism to the canonical model for  $S$ 
of general type,  or a  birational morphism  onto a normal surface $S'$ for $S$ with numerically trivial canonical class. 
Then $f$ is crepant and $S'$ has only rational double points as its singularties. As $S'$ is always embedded into some projective space $\BP^N$ ($N \ge 5$), we may assume that $S' \subset \BP^N$. Then, as the embedded dimension of any rational double point is $3$, it follows that any general linear projection $\pi : \BP^N \dasharrow \BP^4$, as explained above, induces a morphism 
$$\pi|_{S'} : S' \to X:= \pi(S') \subset \BP^4$$ 
such that $\pi|_{S'}$ is isomorphic at all rational double points of $S'$ and $X$ has only these rational double points and  finitely many improper double points. So, the morphism 
$$\pi|_{S'} \circ f : S \to X \subset \BP^4$$ 
satisfies the assumption made in Theorem \ref{thm1}. In fact, this gives, among other things, a generalization of a result of the first author \cite[Prop. 6.2, Cor. 6.3]{Ca97}
see also \cite[Theorem 0.2]{Ca16}, and  an affirmative answer to the conjecture made there. 

\begin{corollary}\label{cor1} Let $S$ be a minimal smooth surface of general type with $p_g(S) = 5$ such that the canonical map  
$$\Phi_{|K_S|} : S \to S'\subset \BP^4$$ is a morphism with
 image isomorphic to  the canonical model $S'$ of $S$. Then, setting  $d : = (K_S^2)_S$,
$$12 \chi({\mathcal O}_S) = (17-d)d\,\, .$$
In particular,  $S'$ is then a complete interesection of type $(2, 4)$ or $(3, 3)$, if either the base field is of characteristic $0$ or the characteristic of the field $k$ is odd and $h^1({\mathcal O}_S) = 0$. 
\end{corollary}

\begin{proof} We have $\delta = 0$, $(H_S^2)_S = (K_S^2)_S = d$ by our assumption. So, the first equality follows from Theorem \ref{thm1}. The last statement then follows from the proof of \cite[Theorem 0.2]{Ca16}. Note that, in positive odd characteristic case, we have 
$h^{1}(mK_S) = 0$ for all $m \ge 1$ by our assumption (for $m=1$) and \cite[Theorem 1.7]{Ek88}  (for $m \ge 2$); hence 
$\chi({\mathcal O}_S) = 6$ by our assumption, and the rest of the proof follows as in \cite{Ca97}, page 41.  

\end{proof}
\begin{problem}
Are there exceptions to the above statement that $S'$ is a complete intersection in odd positive characteristic for $H^1(\hol_S) \neq0$, or in characteristic $p=2$?
\end{problem}

The above  corollary was actually an  initial  motivation for  our study.

The following fact will be used frequently in the proof of Theorem \ref{thm1}.

\begin{lemma}\label{lem0} Let $\pi : V \to W$ be the blow up at $Q \in W$ of a smooth projective variety $W$ of dimension $4$, 
let $E = \pi^{-1}(Q) \simeq \BP^3$ be  the exceptional divisor and let $P \simeq \BP^2$ be a plane in $E \simeq \BP^3$. 
Then, in the Chow ring of $V$, we have:
$$c_1(V) = \pi^*c_1(W) -3E\,\, ,\,\, c_2(V) = \pi^{*}c_2(W) - 2P\,\, ,\,\, (P.P)_V = -1\,\, .$$ 
\end{lemma}

\begin{proof} As $\pi$ is the blow up at a smooth point of a $4$-fold, we have that the Chow ring $A(V)$ of $V$ is generated by the pull back of the Chow ring of $W$
and by the exceptional divisor.

We have moreover
$E^2 = E|_E = -P$. Thus 
$$(P.P)_V = (E^4)_V = ((E|_E)^3)_E = ((-P)^3)_E = -1\,\, .$$
This proves the last formula. As $c_1(V) = -K_V$ and $c_1(W) = -K_W$, the first formula is nothing but the canonical bundle formula under the blow up. 

Let us show the second formula. We have $A^2(V) = \pi^*A^2(W) \oplus \Z [P]$ and,
evaluating on two dimensional cycles avoiding the exceptional divisor, we infer that  we can write  $c_2(V) = \pi^{*}c_2(W) + aP\,\, .$ In order to determine $a$, consider the exact sequence
$$0 \to T_P \to T_V|_{P} \to N_{V/P} \to 0\,\, .$$
By  functoriality of the Chern class, we have
$${\rm (A)}\,\, (c_2(V).P)_V = (c_1(P).c_1(N_{V/P}))_{P} + c_2(P) + c_2(N_{V/P})\,\, .$$
The left hand side is then
$$(c_2(V).P)_V = (\pi^{*}c_2(W) + aP.P) = a(P.P) = -a\,\, .$$
We compute the right hand side. Note that $c_1(P) = -K_P = 3l$ in the Chow ring, where $l$ is a line in $P \simeq \BP^2$. We have 
$${\rm (B)}\,\, c_1(V)|_{P} = c_1(P) + c_1(N_{V/P}) = 3l + c_1(N_{V/P})$$
from the exact sequence above. 
Note that, as $l, P \subset E$, we have 
$$(E|_P.l)_P = (E|_E.l)_E = (-P.l)_E = -1\,\, .$$ 
In particular, $E|_P = -l$ in the Chow ring of $P$. 
Combining this with the first equality, we have
$$c_1(V)|_{P} = (\pi^*c_1(W) - 3E)|_{P} = -3E|_P = 3l\,\, .$$
Hence  $c_1(N_{V/P}) = 0$ by (B). 
Since $P \subset V$ is a smooth subvariety of codimension $2 = 4/2$, by the selfintersection formula and the last formula $(P^2)_V = -1$, we have $c_2(N_{V/P}) = (P^2)_V = -1\,\, .$
As $P \simeq \BP^2$, we have $c_2(P) = 3$. 
Substituting everything into  formula (A), we obtain  $-a = 0 + 3 -1$, hence $a = -2$ as claimed.
\end{proof}

The rest of this section is devoted to prove Theorem \ref{thm1}. 

\begin{proof} Let $x_j$, for $1 \le j \le \delta$,  be the  IDP's  of $X$. Consider the blow up 
$$f : V_1 \to \BP^4$$
of $\BP^4$ at the points $x_j$,  $1 \le j \le \delta$.
  
Let $S_0$ be the proper transform of $X$ and  $f|_{S_0} : S_0 
\to X$ be the induced morphism. Then $f|_{S_0}$ is an isomorphism  except over the points $x_j$,   $S_0$ is smooth over the points $x_j$ and $f|_{S_0}^{-1}(x_j)$ consists of two $(-1)$-curves, i.e., two smooth rational curves of self-intersection number $(-1)$. We denote these two curves by $l_{j1}$ and $l_{j2}$. 

Next, we are going to  take the minimal embedded resolution of $S_0 \subset V_0$ via a sequence of point blow-ups. 
Let 
$$g_1 : V_1 \to V_0$$
be the blow up of $V_0$ at a rational double point, say $P$, of $S_0$ and define $S_1$ to be the proper transform of $S_0$ under $g_1$. Then we have an induced morphism  $g_1|_{S_1} : S_1 \to S_0$ which is the same as the blow up of $S_0$ at the maximal ideal of $P$. The special properties of  rational double points that  we need is that the canonical divisor $K_{S_0}$ is Cartier, $S_1$ continues to have only rational double points as singularities, and the morphism $g_1|_{S_1}$ is crepant, i.e., $(g_1|_{S_1})^*K_{S_0} = K_{S_1}$ (see \cite{Ar66}). Now we choose an embedded resolution of $S_0 \subset V_0$ inductively by
$$g_i : V_{i} \to V_{i-1}\,\, ,\,\, g_i|_{S_{i}} : S_{i} \to S_{i-1}$$
($i=1, 2, \ldots, n$), where $g_i$ is the blow up  of $S_{i-1}$ at some (singular) rational double point, $S_{i}$ is the proper transform of $S_{i-1}$, and $S_{n}$ is smooth. 

We set
$$\tilde{S} := S_n\,\, , \,\, g|_{\tilde{S}} : \tilde{S} \to S_0\,\, ,$$
where $g|_{\tilde{S}}$ is the morphism induced by $g : V_n \to V_0$, the composition of the $g_i$' s. We also denote respective  proper transforms of $l_{j1}$, $l_{j2}$ on $\tilde{S}$  by the same letter. This is harmless for us, as they are disjoint from the exceptional divisors of $g$. 

Clearly $\tilde{S}$ is birational to the original $S$ in our Theorem \ref{thm1}. Recall that we started from the birational morphism $\mu : S \to X \subset \BP^4$. Indeed, by our assumption, the birational map  $\nu := \mu \circ f|_{S_0} \circ g|_{\tilde{S}} : \tilde{S} \to S$ is the contraction morphism of the $2\delta$ (-1)-curves $l_{j1}$, $l_{j2}$.
We denote the exceptional divisor of $g_i : V_i \to V_{i-1}$ by $E_i$ 
and let $P_i \simeq \BP^2$ be a plane of $E_i \simeq \BP^3$. 

\begin{lemma}\label{lem1} Under the  above setting, 
$$S_i = g_i^*S_{i-1} -2P_{i}\,\, ,\,\, (S_i.P_i)_{V_i} = 2$$
in the Chow ring of $V_i$. 
\end{lemma}

\begin{proof} Since $A^2(V_i) = g_i^*A^2(V_{i-1}) \oplus \Z P_i$, one can write  $S_i = g_i^*S_{i-1} + aP_i$
for some integer $a$. Then 
$(S_i.P_i)_{V_i} = a(P_i.P_i)_{V_i} = -a$ by the formula above and Lemma \ref{lem0}. On the other hand, as $g_i$ is the blow up at a rational double point of $S_i$, it follows that $S_i|_{E_i}$ is a plane conic curve in $E_i \simeq \BP^3$, while $P_i$ is a plane in $E_i \simeq \BP^3$. Thus
$$(S_i.P_i)_{V_i} = (S_i|_{E_i}.P_i)_{E_i} = 2\,\, .$$
Therefore $a = -2$ as claimed. \end{proof}
We set 
$$\chi_i := (c_2(V_i).S_i)_{V_i} + (c_1(V_i)|_{S_i}.K_{S_i})_{S_i} + 2(K_{S_i}^2)_{S_i} - (S_i.S_i)_{V_i}$$
for all integers $i$ with  $0 \le i \le n-1$. 

\begin{lemma}\label{lem2} In the above setting,  $12 \chi({\mathcal O}_S) = \chi_{n-1}$.
\end{lemma}

\begin{proof} Since $S_n \subset V_n$ are both smooth, we have the  exact sequence
$$0 \to T_{S_n} \to T_{V_n}|_{S_n} \to N_{V_n/S_n} \to 0\,\, {\color{red} ,}$$
 and we obtain
$${\rm (A):}\,\,\,\, c_1(V_n)|_{S_n} = c_1(S_n) + c_1(N_{V_n/S_n}) = -K_{S_{n}} + c_1(N_{V_n/S_n})\,\, .$$

By Lemma \ref{lem0}, we compute the left hand side of (A) as
$$c_1(V_n)|_{S_{n}} = (g_n|_{S_n})^*(c_1(V_{n-1})|_{S_{n-1}}) -3R_n\,\, ,$$
where $R_n := E_n|_{S_n}$.  Substituting this into (A), we obtain
$${\rm (B):}\,\,\,\, c_1(N_{V_n/S_n}) = (g_n|_{S_n})^*(c_1(V_{n-1})|_{S_{n-1}}) -3R_n + K_{S_{n}}\,\, .$$
Again from the normal bundle sequence, we obtain
$${\rm (C):}\,\,\,\, c_2(V_n)|_{S_n} = c_1(N_{V_n/S_n})c_1(S_n) + c_2(S_n) + c_2(N_{V_n/S_n})\,\, .$$
Using Lemma \ref{lem0}, we compute the left hand side of (C) as
$$c_2(V_n)|_{S_n} = (g_n^*c_2(V_{n-1}).S_n)_{V_n} - 2(P_n.S_n)_{V_n} = (c_2(V_{n-1}).S_{n-1})_{V_{n-1}} -4\,\, .$$

We compute now each term of the right hand side of (C). 
Since $g_n|_{S_n}$ is crepant, we have $(R_n.K_{S_n})_{S_n} = 0$. Thus, by using (B), we compute
$$(c_1(N_{V_n/S_n}).c_1(S_n))_{S_n} = ((g_n|_{S_{n}})^{*}(c_1(V_{n-1})|_{S_{n-1}}). (-K_{S_n}))_{S_n} - (K_{S_{n}}^2)_{S_{n}} = $$
$$= (c_{1}(V_{n-1})|_{S_{n-1}}.-K_{S_{n-1}})_{S_{n-1}} - (K_{S_{n-1}}^2)_{S_{n-1}}\,\, .$$
Using Noether's formula for the smooth projective surface $S_n$, the fact that $\chi({\mathcal O}_S)$ is invariant under birational smooth modification, 
and the fact that $g_n|_{S_n}$ is crepant, we compute
$$c_2(S_n) = 12\chi({\mathcal O}_{S_n}) - (K_{S_n}^2)_{S_n} = 12\chi({\mathcal O}_{S}) - (K_{S_{n-1}}^2)_{S_{n-1}}\,\, .$$
By using again the selfintersection  formula $c_2(N_{V_n/S_n}) = (S_n^2)_{V_n}$ and the formula $S_n = g_n^*S_{n-1} - 2P_n$ (Lemma \ref{lem1}), we have
$$c_2(N_{V_n/S_n}) = (S_n.S_n)_{V_n} = (S_{n-1}.S_{n-1})_{V_{n-1}} -4\,\, .$$
Substituting all into  equation (C), we obtain that 
$$(c_2(V_{n-1}).S_{n-1})_{V_{n-1}} -4 = (c_{1}(V_{n-1})|_{S_{n-1}}.-K_{S_{n-1}})_{S_{n-1}} - 2(K_{S_{n-1}}^2)_{S_{n-1}}$$ 
$$+ (S_{n-1}.S_{n-1})_{V_{n-1}} -4 + 12 \chi({\mathcal O}_S)\,\, .$$
By the definition of $\chi_i$, this equality is equivalent to  $12 \chi({\mathcal O}_S) = \chi_{n-1}$. This completes the proof of Lemma \ref{lem2}. 

\end{proof}

\begin{lemma}\label{lem3} The number  
$\chi_i$ is independent of  $i$, for all $0 \le i \le n-1$. 
\end{lemma}

\begin{proof} Let $i \ge 1$. Then by Lemma \ref{lem0} and Lemma \ref{lem1}, we have
$$(c_2(V_i).S_i)_{V_i} = ((g_i^*c_2(V_{i-1}) - 2P_i).(g_i^*S_{i-1} - 2P_i))_{V_i} 
= (c_2(V_{i-1}).S_{i-1})_{V_{i-1}} - 4\,\, ,$$
$$(S_i.S_i)_{V_i} = ((g_i^*S_{i-1} - 2P_i).(g_i^*S_{i-1} - 2P_i))_{S_i} = (S_{i-1}.S_{i-1})_{V_{i-1}} -4\,\, .$$ 
Using Lemma \ref{lem0} and observing that $K_{S_i} = (g_{i}|_{S_i})^*K_{S_{i-1}}$, we have 
$$(c_1(V_i)|_{S_i}.K_{S_i})_{S_i} = ((g_i^*c_1(V_{i-1}) - 2P_i)|_{S_i}.(g_{i}|_{S_i})^*K_{S_{i-1}}) = (c_1(V_{i-1})|_{S_{i-1}}.K_{S_{i-1}})_{S_{i-1}}\,\, ,$$
and  $(K_{S_i}^2)_{S_i} = (K_{S_{i-1}}^2)_{S_{i-1}}$.
Substituting all these equalities  into the formula in Lemma \ref{lem3}, we obtain that  $\chi_i = \chi_{i-1}$. This implies the result.
 
\end{proof}

\begin{lemma}\label{lem4} In the  above setting, 
$$12 \chi({\mathcal O}_S) = (c_2(V_0).S_0)_{V_0} + (c_1(V_0)|_{S_0}.K_{S_0})_{S_0} + 2(K_{S_0}^2)_{S_0} - (S_0.S_0)_{V_0}\,\, .$$
\end{lemma}

\begin{proof} By Lemma \ref{lem2} and Lemma \ref{lem1}, we have  $12 \chi({\mathcal O}_S) = \chi_{n-1} = \chi_{0}$. The right  hand side of the formula in Lemma \ref{lem4} is by  definition equal to $\chi_{0}$ . This completes the proof. \end{proof}

\begin{lemma}\label{lem5} In the  above situation, and with the usual notation where 
$H_S$ is the pullback of the hyperplane class $H$ of $\BP^4$ under $\mu : S \to X \subset \BP^4$ and $d = (H_S^2)_{S} = {\rm deg}\, X$,
we have:
$$12 \chi({\mathcal O}_S) = 10d - d^2 + 2(K_S^2)_{S} + 5(H_S.K_S) + 2\delta\,\, .$$
 
\end{lemma}

\begin{proof} We shall prove the formula by computing each term of the right hand side of the equation in Lemma \ref{lem4}. 

Let $F_j \simeq \BP^3$ be the exceptional divisor of $f : V_0 \to \BP^4$ over $x_j$ and let $R_j \simeq \BP^2$ be a plane in $F_j$. 

Then  $c_{*}(\BP^4) = (1+H)^5$ in the Chow ring of $\BP^4$, and
$$c_1(V_0) = f^*c_1(\BP^4) - 3 \sum_{j=1}^{\delta} F_j\,\, ,$$
$$c_2(V_0) = f^*c_2(\BP^4) - 2 \sum_{j=1}^{\delta} R_j\,\, ,\,\, (R_j.R_j)_{V_0} = -1$$
by Lemma \ref{lem0}, and 
$$S_0 = f^*X - 2 \sum_{j=1}^{\delta} R_j = f^*(dH^2) - 2 \sum_{j=1}^{\delta} R_j$$
in the Chow ring of $V_0$.  Hence 
$$(c_2(V_0).S_0)_{V_0} = ((10f^*H^2 - 2\sum_{j=1}^{\delta} R_j).(df^*H^2 - 2\sum_{j=1}^{\delta} R_j)_{V_0} = 10d - 4\delta\,\, ,$$ 
$$(S_0.S_0) = ((df^*(H^2) - 2\sum_{j=1}^{\delta} R_j)^2) = d^2 - 4 \delta\,\, .$$
As $g : \tilde{S} \to S_0$ is crepant and $\tilde{S} \to S$ contracts exactly $2\delta$ disjoint $(-1)$-curves, we obtain 
$$(K_{S_0}^2)_{S_0} = (K_{\tilde{S}}^2)_{\tilde{S}} = (K_S^2)_{S} - 2\delta\,\, .$$

Finally we compute $(c_1(V_0)|_{S_0}.K_{S_0})_{S_0}$. First of all, we have
$$(c_1(V_0)|_{S_0}.K_{S_0})_{S_0} = ((f \circ g)^*c_1(\BP^4)|_{\tilde{S}}.K_{\tilde{S}})_{\tilde{S}}$$ 
$$= (5H_{\tilde{S}} - 3\sum_{j=1}^{\delta}F_j|_{\tilde{S}}).K_{\tilde{S}})_{\tilde{S}} 
= (5H_{\tilde{S}}.K_{\tilde{S}})_{\tilde{S}} - 3\sum_{j=1}^{\delta} (l_{j1} + l_{j2}).K_{\tilde{S}})_{\tilde{S}} \,\, .$$
Here $F_j$ is the proper transform of the exceptional divisor $f^{-1}(x_j)$ of $f$ and we used the fact that $g|_{\tilde{S}} : \tilde{S} \to S_0$ is crepant for the second equality so there appears no exceptional divisor of $g_i$ in the formula. As $l_{j1}$ and $l_{j2}$ are $(-1)$-curves on a smooth surface $\tilde{S}$, we have 
 $(l_{j1}.K_{\tilde{S}})_{\tilde{S}} = (l_{j2}.K_{\tilde{S}})_{\tilde{S}} = -1$. Moreover, since the morphism $\nu : \tilde{S} \to S$ defined above is the contraction of exactly $2\delta$ $(-1)$-curves $l_{j1}$ and $l_{j2}$, it follows that
$$(H_{\tilde{S}}.K_{\tilde{S}})_{\tilde{S}} = (\nu^{*}H_{S}.(\nu^{*}K_S + \sum_{j=1}^{\delta} (l_{j1} + l_{j2})))_{\tilde{S}} = (H_S.K_S)_{S}\,\, .$$ 
Substituting these two formulae into the last formula for $(c_1(V_0)|_{S_0}.K_{S_0})$, we obtain that  $(c_1(V_0)|_{S_0}.K_{S_0}) = 5(H_S.K_S)$. 
Substituting the four formulae that we obtained  for $(c_2(V_0).S_0)_{V_0}$, $(c_1(V_0)|_{S_0}.K_{S_0})_{S_0}$, $(K_{S_0}^2)_{S_0}$ and $(S_0.S_0)_{V_0}$ inside 
the formula in Lemma \ref{lem4}, we obtain 
$$12\chi({\mathcal O}_S) = 10d - 4\delta + 5(H_S.K_S) + 6\delta + 2(K_S^2) - 4\delta -d^2 + 4\delta\,\, .$$
Simplifying the right hand side, we obtain
$$12 \chi({\mathcal O}_S) = 10d - d^2 + 2(K_S^2)_{S} + 5(H_S.K_S)_S + 2\delta\,\, ,$$
as claimed.  
\end{proof}

This completes the proof of Theorem \ref{thm1}. \end{proof}

\section{Double point formulae via symplectic approximations.}

We introduce, inspired by  a concept introduced by Kodaira in \cite{kod}, a class of isolated singularities.

\begin{defin}
An n-dimensional  isolated singularity $O \in X \subset \CC^{2n}$ is said to be {\bf quasi-improper-multiple point = QIMP},
if $X$ consists of $r$ smooth branches $X_1, \dots, X_r$ passing through the origin $O$.
\end{defin}

To explain the notion, one can take a good projection yielding a splitting $\CC^{2n} = \CC^n \oplus \CC^n$,
so that 
$$ X_i = \{ (x,y) \in  \CC^n \oplus \CC^n | y = F_i (x) \}.$$

Clearly the intersection points $X_i \cap X_j$ correspond to the set
$$ \Ga_{ij} : = \{ x \in \CC^n | F_i (x) - F_j(x) = 0 \}.$$

The hypothesis of isolated singularities amounts to the requirement that $0 \in \CC^n$ is isolated in the locus
$ \Ga_{ij}$, and, setting $\hol : =  \hol_{\CC^n, 0}$, we consider the intersection multiplicity  
$$  m_{ij} : = dim_{\CC}( \hol / (F_i - F_j)\hol^n),$$
 where $F_i -F_j$ is considered as a $1 \times n$ matrix.
It is clear that, for a generic perturbation of  the branches $X_i$, the isolated singularity deforms to 
$$\de_O : = \sum_{i<j} m_{ij} $$
IDP's. This is why we shall say that $\de_O$ is the {\bf local number of IDP's}.

\begin{theo}\label{symplgen}
Let $X \subset \PP^{2n}$ be a  complex variety with isolated singularities, of which 

(1) $h$ are quasi-improper-multiple points,
such that the sum of the local numbers of IDP's equals  $\de$,

(2)  the other  singular points are normal  and locally smoothable.  

Then $X$ admits a global smoothing to a
symplectic immersed manifold $M \subset \PP^{2n}$, with exactly  $\de$   IDP's,
and we have, if $f : M' \ra M$ is the immersion,  that
$$ d^2 = 2 \de + e(N_f) ,$$
where $e$ is the Euler class of the oriented normal bundle to the map.
\end{theo}

\begin{proof}

The result follows from  Theorem 1.2 of \cite{Ca09}, and a local deformation at the 
non normal singularities, showing that we can deform $X$ to 
a symplectic immersed manifold $M \subset \PP^{2n}$, with exactly  $\de$   IDP's.

Letting $f : M' \ra M$ be the immersion, we calculate the self intersection $d^2$ of $M$ in $\PP^{2n}$
(in the same way as explained in the introduction)
as $ 2 \de$ plus the self intersection number of the zero section in the normal bundle $N_f$ (which maps via the exponential map
onto a neighbourhood of $X$). \end{proof} 

\begin{remark}
A special case is the one where the normal singularities are isolated hypersurface singularities.

The following was the original idea which led us to realize that the Severi double point formula
 for smooth surfaces (i.e., smooth outside of the IDP's) holds verbatim if we also allow rational double points,  at least over $\mathbb C$.

In fact, for surface rational double points  over $\mathbb C$, one has the fortunate coincidence that $M'$ coincides with the minimal resolution
$\pi : S \ra X$, and then one has only to observe   that $\pi$ and $f$ are a differentiable deformation of each other,
which can be taken as the identity outside the inverse image $S_B$ of  a sufficiently small neighbourhood $B$ of the 
normal singularities.
 
In this  case, a possible way to show the equality between $e(N_f)$ and $c_2 (N_{\pi})$,   following the arguments of \cite{Ca09},  could 
be to argue  that  the Euler number of $N_f$ is 
obtained by integrating the top  Chern form   of   $N_{\pi}$ outside  $S_B$,
plus adding an integral on the Milnor fibre $M'_B$,  diffeomorphic to the neighbourhood $S_B$ of the exceptional divisor:
this integral should yield the top Chern class integral on the tangent bundle of $S_B$, and  then  the desired equality would follow.

\end{remark}

{\bf Acknowledgements}:  The authors  would like to thank Alex Dimca for useful discussions, and Paolo Aluffi for 
answering the first author's email queries.

\end{document}